\documentclass[11pt,a4paper]{article}
\usepackage[latin1]{inputenc}
\usepackage[T1]{fontenc}
\usepackage{amsmath}
\usepackage{amssymb}
\usepackage{amsthm}
\newtheorem{theorem}{Theorem}
\newtheorem{lemma}{Lemma}
\newtheorem{corollary}{Corollary}
\usepackage{natbib}
\usepackage{bm}

\usepackage[english]{babel}
\usepackage{float}
\begin{document}
\begin{center}
  \Large
  \textbf{Optimal designs for two-level main effects models on a
  restricted design region}
\end{center}
\begin{center}
  \textbf{Fritjof Freise\footnote{corresponding author}\\
    TU Dortmund University, Department of Statistics,\\
    Vogelpothsweg 87, 44227 Dortmund, Germany,\\
    e-mail: fritjof.freise@tu-dortmund.de}
\end{center}
\begin{center}
  \textbf{Heinz Holling\\
  University of M\"unster, Institute for Psychology,\\
  Fliednerstra{\ss}e 21, 48149 M\"unster, Germany,\\
  e-mail: holling@wwu.de}
\end{center}
\begin{center}
  \textbf{Rainer Schwabe\\
    University of Magdeburg, Institute for Mathematical Stochastics,\\
    Universit\"atsplatz 2, 39106 Magdeburg, Germany,\\
    e-mail: rainer.schwabe@ovgu.de}
\end{center}

\begin{abstract}
  We develop $D$-optimal designs for linear main effects models on a
  subset of the $2^K$ full factorial design region, when the number of
  factors set to the higher level is bounded.  It turns out that in
  the case of narrow margins only those settings of the design points
  are admitted, where the number of high levels is equal to the upper
  or lower bounds, while in the case of wide margins the settings are
  more spread and the resulting optimal designs are as efficient as a
  full factorial design.  These findings also apply to other
  optimality criteria.
\end{abstract}


\section{Introduction}
The motivation for this work comes from the problem of calibration of
items in educational and psychological tests.  These items are
constructed using a number of rules which may be either applied or
not.  In calibration experiments items are presented to a large number
of individuals with essentially known ability.  The item parameters
describe the influence of each rule on the mean score of the
individuals and are to be estimated by the responses in the
calibration experiment.

A restriction arising in this scenario is, that items become more
difficult, if the number of active rules is increased.  Hence, from a
practical point of view it would not make much sense to use only one
rule or no rule at all, because the item would be too easy.  On the
other hand the item would become too difficult, if the number of
active rules is too large.  For those items with too few or too many
rules it would be doubtful that a linear model assumed is valid.  This
matter imposes a restriction on the design region, which allows only
such items with a bounded number of active rules.  For example, in an
experiment with six different rules it would be meaningful to restrict
to items with at least two and at most four active rules.  In
particular, under these constraints neither the full factorial nor
regular fractional factorial designs can be used any more.

In the present paper we will consider a linear model providing the
fundamentals of the so-called classical test
theory~\citep[e.g.][]{McDonald:1999} which are still mostly used for
the development and calibration of educational and psychological
tests.

We will start in Section~\ref{sec:Examples} by briefly outlining rule
based item generation. Then, in Section~\ref{sec:ModelInfoInvDesigns}
the model and its information matrix are presented, which is the basis
for the comparison of designs.  After a short introduction to optimal
design and invariance, the special structure of the information matrix
is discussed in Section~\ref{sec:StructureOfInfo}.  This is followed
in the subsequent section by the main results, which constitute
conditions on designs for $D$-optimality.  Proofs and exemplary tables
of designs are deferred to appendices.

Our results are related to work in spring balance weighing and
chemical balance weighing designs. In contrast to the model considered
here these usually do not incorporate constants. For $D$- and
$A$-optimal designs with restrictions on the number of objects used in
each weighing see~\citet{HudaMukerjee:1988}. Optimality for spring
balance designs without restrictions but including a constant in the
model is considered in~\citet*{FilovaHarmanKlein:2011a}.
\section{Rule-based item generation}
\label{sec:Examples}

Items of educational and psychological tests should neither be too
easy nor too difficult. Otherwise, ceiling or bottom effects may
occur, i.e.\ the number of correctly solved items of several
respondents reach the maximal or minimal value. The information of
such items may be severely impaired. Furthermore, according to a
well-known result of classical test theory mean item difficulties are
desirable since they foster high item discriminations, i.e.\ high
correlations between item scores and the total scores.

For rule-based item generation~\citep[e.g.][]{ArendasySommer:2007},
an efficient method for item development, this objective can be best
achieved by items with an appropriate number of rules. The rules are
usually represented by particular demands on cognitive processing and
determine the item parameters, mostly item difficulty. Such item
generation will be briefly illustrated by items measuring mental speed
for numerical operations.

Each item may consist of a comparable set of 20 numbers with 4
digits, such as 3412, 5364, 2774, $\hdots$, 8732. These numbers are
nowadays often generated on the fly and displayed on a computer
screen for a certain amount of time. For items with only one rule
respondents with high ability will mark most numbers correctly if not
all. On the other hand, difficult items characterised by 6 of more
rules will lead to low scores especially for respondents with low
abilities.

Another example for rule based generated items which measure human
memory are represented by sets of stimuli which are defined by binary
characteristics (rules) and generated according to a full factorial
design. Two basic elements, e.g., circles and triangles are
furthermore characterised by a set of binary attributes, for example,
colour, size or shading. Again, the difficulty of these items is
mainly determined by the number of the attributes (rules). These items
will be displayed to the respondents for a certain period of time and
have to be recognised by the respondents some time later.

In general, when a set of rule-based items will be presented to a
sample of respondents, items with a too small or a too large number of
rules should not be used in order to avoid ceiling and bottom
effects. Furthermore, many respondents will not obtain item scores
near to the mean score. Hence, to estimate the influence of the rules
on the difficulty of the items by linear models the design region has
to be restricted.

\section{Model, Information and Design Invariance}
\label{sec:ModelInfoInvDesigns}

We consider an experiment in which $N$ items are presented and
responses $Y_1, \ldots, Y_N$ are observed.  The number of rules, which
are used to construct the items, is $K$.  Then the items can be
characterised by the corresponding design points
$\bm{x}_i=(x_{i1}\,\ldots,x_{iK})^\top\in\{-1,+1\}^K$, where the
entries $x_{ij}$ are equal to $+1$, if the $j$-th rule is used in the
construction of the $i$-th item, and $x_{ij}=-1$, if the rule is not
used.  We assume that only main effects occur and that there are no
interactions between the rules.  Then the difficulties of the items
are specified by the parameter vector
$\bm{\beta}=(\beta_0,\beta_1,\ldots,\beta_K)^\top\in\mathbb{R}^p$,
where the number of parameters equals $p = K + 1$, and which includes
a constant term $\beta_0$ besides $K$ parameters $\beta_j$,
$j=1,\ldots,K$, corresponding to the main effects of the $K$ rules.
The model can be written in a general linear model equation as
\begin{equation*}
  Y_{i}=\bm{f}(\bm{x}_{i})^\top\bm{\beta}+\varepsilon_{i}\,,
\end{equation*}
$i=1,\ldots,N$, with regression function
$\bm{f}(\bm{x})=(1,\bm{x}^\top)^\top=(1,x_{1},\ldots,x_{K})^\top$.  As
usual in linear models it is assumed that the errors
$\varepsilon_{1},\ldots,\varepsilon_{N}$ are uncorrelated and
homoscedastic with mean $\mathrm{E}(\varepsilon_{i})=0$ and variance
$\mathrm{Var}(\varepsilon_{i})=\sigma^{2}$.  By letting
$ \bm{Y}=(Y_{1},\ldots,Y_{N})^\top$ and
$ \bm{\varepsilon}=(\varepsilon_{1},\ldots,\varepsilon_{N})^\top$ the
vector of observations and errors, respectively, and
$ \mathbf{F}=(\bm{f}(\bm{x}_{1}),\ldots,\bm{f}(\bm{x}_{N}))^\top$ the
design matrix, the model can be written in vector notation
\begin{equation*}
  \bm{Y} = \mathbf{F} \bm{\beta}+\bm{\varepsilon}\,.
\end{equation*}

In what follows we consider the situation that the design region
$\mathcal{X}\subseteq \{-1,+1\}^K$ is restricted by the possible
number of active rules, i.e.\ the number of factor levels $+1$ in a
design point $\bm{x}$, which is given by
$ d(\bm{x})= (K+\sum_{j=1}^{K}x_{j})/2$.  We denote the minimal and
maximal number of active rules by $L$ and $U$, respectively, and
assume $L < U$. This excludes the case $L=U$, in which the design
matrix $\mathbf{F}$ would not have full column rank and hence
$\bm{\beta}$ could not be estimated. The design region is then
specified by
\begin{equation*}
  \mathcal{X}
  = \{\bm{x}\in\{-1,+1\}^K \,|\, L \leq d(\bm{x}) \leq U\}
  = \{\bm{x}\,|\, 2L-K \leq
    \sum_{j=1}^{K}x_{j}
    \leq 2U-K\}\,.
\end{equation*}
The (unrestricted) full factorial design region would be given by
$L=0$ and $U=K$.

It is well-known (see e.g.~\citealp{Searle:1971}, p. 90, or
\citealp{Rao:1973}, p. 226), that for $\mathbf{F}$ with full column
rank $\bm{\beta}$ is estimable and the variance of the least squares
estimator is proportional to the inverse of the information matrix
\begin{equation*}
   \mathbf{F}^\top\mathbf{F} = \sum_{i=1}^{N}\bm{f}(\bm{x}_{i})\bm{f}(\bm{x}_{i})^\top\,.
\end{equation*}

To facilitate the search for optimal designs we will make use of
approximate design theory~\cite[see for example][]{Silvey:1980}.  In
this context an approximate design $\xi$ is defined by
\begin{equation*}
  \xi=\left\{
    \begin{matrix}
      \bm{x}_{1}&\hdots&\bm{x}_{n}\\
      w_{1}&\hdots&w_{n}
    \end{matrix}
  \right\}\,,
\end{equation*}
where $\bm{x}_{1},\hdots,\bm{x}_{n}\in\mathcal{X}$ are mutually
distinct settings with weights $w_{i} \geq 0$, $i = 1,\ldots, n$,
$\sum_{i=1}^{n}w_{i} = 1$.  Here the weights $w_i$ represent the
proportions $\xi(\bm{x}_i)$ of observations, which should be spent at
$\bm{x}$.

The corresponding (weighted) information matrix is defined as
\begin{equation*}
  \mathbf{M}(\xi)
  = \sum_{i=1}^{n}w_{i}\bm{f}(\bm{x}_{i})\bm{f}(\bm{x}_{i})^\top\,,
\end{equation*}
which is the average information per observation.  In the case of an
exact design $(\bm{x}_1,\ldots,\bm{x}_N)$ for $N$ observation the
weights $w_{i}$ equal $N_{i}/N$, where $N_{i}$ is the number of
replications at $\bm{x}_{i}$, and the weighted information matrix
equals $ N^{-1} \mathbf{F}^\top\mathbf{F} $.

A design $\xi^*$ on $\mathcal{X}$ is $D$-optimal if and only if it
maximises the determinant of the information matrix, i.e.\
\begin{equation*}
  \det(\mathbf{M}(\xi^*)) \geq \det(\mathbf{M}(\xi))
\end{equation*}
for all designs $\xi$ on $\mathcal{X}$.  Under the $D$-criterion
the volume of the confidence ellipsoid for $\bm{\beta}$ is minimised.

We will use invariance properties to reduce the complexity of the
optimisation problem. See for example~\cite{Pukelsheim:1993}
and~\cite{Schwabe:1996} for details and further references.  This
approach is also used in~\cite{FilovaHarmanKlein:2011a} to derive
results on $E$-optimal spring balance weighing designs.  In this
context the design region consists of the vertices of the $K$
dimensional unit cube $\{0,1\}^{K}$ and no restriction on the number
of active levels $1$ is considered.

The design region $\mathcal{X}$, considered as a subset of the
vertices of the hypercube $\{-1,+1\}^{K}$, is invariant under
permutations of the entries in the design points, i.e.\ permutations
of rules, which result in appropriate rotations of the hypercube.
Under the group of these permutations there are $U-L+1$ orbits, which
will be denoted by $\mathcal{O}_{k}$, $k=L,\ldots,U$.  The orbit
$\mathcal{O}_{k}=\{\bm{x}|d(\bm{x})=k\}$ consists of all items with
$k$ active rules or, equivalently, of the design points with $k$
entries equal to $+1$, i.e.\
\begin{equation*}
  \mathcal{O}_{k}=\{\bm{x}\in\mathcal{X} \,|\, \sum_{j=1}^{K}x_{j}=2k-K\}\,.
\end{equation*}
Note, that the orbits yield a partition of the design region
$\mathcal{X}$, i.e.\ they are mutually disjoint and
$\mathcal{X}=\bigcup_{k=L}^U{\mathcal{O}_{k}}$.  The present main
effects model is linearly equivariant with respect to permutations,
i.e.\ for each permutation $\mathbf{P}$ exists a matrix $\mathbf{Q}$
such that $\bm{f}(\mathbf{P}\bm{x}) = \mathbf{Q}\bm{f}(\bm{x})$
uniformly in $\bm{x}$.  A design $\bar{\xi}$, which remains unchanged,
if the support is transformed, here by permutation of rules, is called
invariant.  From the equivariance of the model and the invariance of
the design region follows that there exists an invariant $D$-optimal
design.

These invariant designs have uniform weights on each orbit, i.e.\ for
all $\bm{x}_{1},\bm{x}_{2}\in \mathcal{O}_{k}$ follows
$\bar{\xi}(\bm{x}_{1}) = \bar{\xi}(\bm{x}_{2})$.  Denote the uniform
design on the orbit $\mathcal{O}_k$ with $k$ rules by $\bar{\xi}_{k}$.
These are called vertex designs in~\cite{FilovaHarmanKlein:2011a}.
For the invariant design $\bar\xi_k$ on $\mathcal{O}_k$ the weights
$\xi_k(\bm{x})=1/{K \choose k}$ are determined as the reciprocal of the
number ${K \choose k}$ of design points in the orbit, and the
information matrix is given by
\begin{equation*}
  \mathbf{M}(\bar{\xi}_{k}) = {K \choose k}^{-1}
  \sum_{\bm{x}\in \mathcal{O}_{k}}\bm{f}(\bm{x})\bm{f}(\bm{x})^\top\,.
\end{equation*}

Every invariant design $\bar\xi$ can be written as a weighted sum of
vertex designs, $\bar\xi=\sum_{k=L}^{U}\bar{w}_k\bar\xi_k$ with
weights $\bar{w}_{k} \geq 0$, $\sum_{k=L}^{U}\bar{w}_{k} = 1$, for the
orbits.  Then the information matrix of an invariant design
$\bar{\xi}$ on $\mathcal{X}$ equals
\begin{equation*}
  \mathbf{M}(\bar{\xi})
  =\sum_{k=L}^{U}\bar{w}_{k}\mathbf{M}(\bar{\xi}_{k})\,.
\end{equation*}
Hence the optimisation can be confined to finding the optimal weights
$\bar{w}_{k}$.  Each invariant design can be characterised by the
orbits $\mathcal{O}_{k}$ and their corresponding weights
$\bar{w}_{k}$. Due to this fact we can use the notation
\begin{equation*}
  \bar{\xi}= \left\{\begin{matrix}
      k_{1}&\cdots&k_{n}\\
      \bar{w}_{1}&\cdots&\bar{w}_{n}
    \end{matrix}\right\}
\end{equation*}
$k_{i}\in\{L,\ldots,U\}$, $i=1,\ldots,n$, whenever an invariant design
on $n$ orbits is given explicitly, where only orbits with non-zero
weights $\bar{w}_k$ are specified.

In the particular case that the constraints are symmetric, $L+U=K$,
i.e.\ whenever items with $k$ active rules are allowed then so are
those with $K-k$ active rules, then invariance additionally is present
with respect to the sign change of the whole vector of the design
point, i.e.\ switching from $k$ active rules to $k$ inactive rules.
Consequently it follows in this case, that there is an invariant
$D$-optimal design with $\bar{w}_{k}=\bar{w}_{K-k}$.
\section{Structure of the Information Matrix}
\label{sec:StructureOfInfo}
The entries of the weighted information matrix are moments with
respect to the design $\xi$, with the general form
\begin{equation*}
  \sum_{i=1}^{n}w_{i}x_{ik}^{u}x_{i\ell}^{v}\,,
  \quad k, \ell \in\{1,\ldots,K\},\,
  u,v\in\{0,1\}.
\end{equation*}
For an invariant design $\bar{\xi}$ the moments reduce to only three
different values.  The diagonal, with the constant ($u=v=0$) and
second moments ($k=\ell$, $u=v=1$), is given by
\begin{equation*}
  \sum_{i=1}^{n}w_{i} = 1\quad\text{and}\quad\sum_{i=1}^{n}w_{i}x_{ik}^{2}  = 1,\quad k=1,\hdots,K\,.
\end{equation*}
The off-diagonal entries in the first row and first column, i.e.\ the
first moments ($u=0$, $v=1$ or vice versa), are identical
\begin{equation*}
  m_{1}(\bar{\xi})= \sum_{i=1}^{n}w_{i}x_{ik}\,,\quad k=1,\ldots,K\,,
\end{equation*}
while all other off-diagonal entries, i.e.\ the mixed moments
($k\neq\ell$, $u=v=1$), also coincide
\begin{equation*}
  m_{2}(\bar{\xi})=\sum_{i=1}^{n}w_{i}x_{ik}x_{i\ell}\,,\quad 1\leq k < \ell \leq K\,.
\end{equation*}
Denoting the $K$-dimensional vector of ones by $\bm{1}_{K}$ and the
$K\times K$ identity matrix by $\mathbf{I}_{K}$, the information
matrix becomes
\begin{equation}
  \label{eq:InfoMatrixStructure}
  \mathbf{M}(\bar{\xi}) = \begin{pmatrix}
    1 & m_{1}(\bar{\xi}) \bm{1}_{K}^\top\\
    m_{1}(\bar{\xi}) \bm{1}_{K} & \mathbf{M}_{22}(\bar{\xi})\\
  \end{pmatrix}
\end{equation}
with the submatrix
\begin{equation*}
  \mathbf{M}_{22}(\bar{\xi}) = (1 - m_{2}(\bar{\xi})) \mathbf{I}_{K} + m_{2}(\bar{\xi})
  \bm{1}_{K}\bm{1}_{K}^\top\,.
\end{equation*}
Even though $m_{1}(\bar{\xi})$, $m_{2}(\bar{\xi})$ and hence
$\mathbf{M}_{22}(\bar{\xi})$ depend on the design $\bar{\xi}$, we will
omit the argument for the sake of brevity, where it does not lead to
confusions.

As we have seen before, the information matrix of an invariant design
can be written as a weighted sum of the information matrices of the
orbits.  These matrices $\mathbf{M}(\bar{\xi}_{k})$ have the same
structure as in~\eqref{eq:InfoMatrixStructure}, with the moments
replaced by the moments of the $k$-th orbit,
\begin{equation}
  \label{eq:def:mbar1k}
  m_{1}(\bar{\xi}_{k})
  = \frac{2 k - K}{K}
\end{equation}
and
\begin{equation}
  \label{eq:def:mbar11k}
  m_{2}(\bar{\xi}_{k})
  = \frac{(2 k - K)^2  - K}{K(K - 1)}\,,
\end{equation}
which can be calculated by counting the number of summands equal to
$+1$ or $-1$.  Some properties following from~\eqref{eq:def:mbar1k}
and~\eqref{eq:def:mbar11k} we will need later are
\begin{align}
  \label{eq:PropOfmbar1k1}
  m_{1}(\bar{\xi}_{k}) \leq 0
  & \quad\mbox{if and only if}\quad k \leq \frac{K}{2}\\
  \intertext{and}
  \label{eq:PropOfmbar11k1}
  m_{2}(\bar{\xi}_{k}) \leq 0
  &\quad\mbox{if and only if}\quad \frac{K-\sqrt{K}}{2}\leq k \leq \frac{K+\sqrt{K}}{2}\,.
\end{align}
Equality holds on the left-hand side of~\eqref{eq:PropOfmbar1k1} if
and only if equality holds on the right-hand side.  Analogously
$m_{2}(\bar{\xi}_{k}) = 0$ in~\eqref{eq:PropOfmbar11k1} if and only if
equality holds for one of the relations on the right-hand side of the
condition.  Note, that due to symmetry
\begin{equation}
  \label{eq:PropOfmbark3}
  m_{1}(\bar{\xi}_{k}) = -m_{1}(\bar{\xi}_{K-k})
  \quad\text{and}\quad
  m_{2}(\bar{\xi}_{k}) = m_{2}(\bar{\xi}_{K-k})\,.
\end{equation}

For further calculations also note that
\begin{equation}
  \label{eq:mAsSumOfmbark}
  m_{1}=\sum_{k=L}^{U}\bar{w}_{k}m_{1}(\bar{\xi}_{k})
P  \quad\text{and}\quad
  m_{2}=\sum_{k=L}^{U}\bar{w}_{k}m_{2}(\bar{\xi}_{k})\,.
\end{equation}

For designs, which are also invariant with respect to sign change, it
follows that $\bar{w}_{k}=\bar{w}_{K-k}$ and hence $m_{1} = 0$.  This
can be easily seen from~\eqref{eq:mAsSumOfmbark}
and~\eqref{eq:PropOfmbark3}.  In fact $m_{1}=0$ holds for all
symmetric invariant designs, i.e.\ $\bar{w}_{k}=\bar{w}_{K-k}$, for
all $k=L,\ldots,U$.  In an invariant design with at least two orbits a
necessary condition for $m_{1}=0$ is, that there are
$k,\ell\in\{L,\ldots,U\}$, with $\bar{w}_{k}>0$ and
$\bar{w}_{\ell}>0$, such that $k<K/2<\ell$. This follows
from~\eqref{eq:PropOfmbar1k1}. Analogously follows
from~\eqref{eq:PropOfmbar11k1} that $m_{2}=0$ only if either
\begin{equation*}
  k\in \left(\frac{K-\sqrt{K}}{2},\frac{K+\sqrt{K}}{2}\right)
  \quad\text{and}\quad
  \ell\notin \left(\frac{K-\sqrt{K}}{2},\frac{K+\sqrt{K}}{2}\right)
\end{equation*}
or, if the boundaries of the given interval are integers, the design
consists of the two orbits corresponding to $(K-\sqrt{K})/2$ and
$(K+\sqrt{K})/2$ only.
\section{Invariant $D$-optimal Designs}
\label{sec:DoptDesign}
As stated in the previous section, we can confine the optimisation on
finding optimal weights
$\bar{w}_{L}, \bar{w}_{L+1}, \ldots, \bar{w}_{U}$.  The determinant of
the information matrix can be calculated using some standard results
on determinants:
\begin{align}
  \det(\mathbf{M}(\xi))
  &=\det\left(\mathbf{M}_{22} - m_{1}^{2} \bm{1}_{K}\bm{1}_{K}^\top\right)
    \notag\\
  &=\det\left(
    (1-m_{2})\mathbf{I}_{K}+(m_{2}- m_{1}^{2})\bm{1}_{K}\bm{1}_{K}^\top
    \right)
    \notag\\
  &= (1-m_{2})^{K-1}\left(1 + (K-1)m_{2} - Km_{1}^2\right)\,.
    \label{eq:DetOfM}
\end{align}
As a direct consequence conditions for the regularity of the
information matrix follow:
\begin{lemma}
  \label{lem:regularInfo}
  For an invariant design $\bar{\xi}$ the information matrix
  $\mathbf{M}(\bar{\xi})$ is regular if and only if there exist
  $k,\ell\in\{L,\ldots, U\}$, $k\neq \ell$, such that $\bar{w}_{k}>0$
  and $\bar{w}_{\ell}>0$, and, for $K\geq 2$, either $k$ or $\ell$ is
  strictly between $0$ and $K$.
\end{lemma}

As we will see shortly, there are two different cases for optimal
invariant designs: Either $m_{1}=m_{2}=0$ or its complement.  The
first case holds if and only if $(K-2L)(2U-K)\geq K$.  The
corresponding information matrix is the identity matrix and hence the
designs are as efficient as the $2^K$ full factorial design. Even for
the unrestricted design region
$\{-1,+1\}^K$. Theorem~\ref{thm:DoptEffAsFactorial} shows the
corresponding result.

If on the other hand $(K-2L)(2U-K) < K$, invariant designs on the
boundary orbits $\mathcal{O}_{L}$ and $\mathcal{O}_{U}$ are
optimal. In fact $D$-optimal designs have to be concentrated on these
two orbits. This can happen if the interval $[L,U]$ is too narrow or
does not include $K/2$.  In those cases it follows that $m_{1}\neq 0$
or $m_{2}\neq 0$ from~\eqref{eq:PropOfmbar1k1}
and~\eqref{eq:PropOfmbar11k1}. (See Lemma~\ref{lem:m1m2not0} in the
Appendix)

The weight $\bar{w}_{L}^*$ in the following Theorem maximises the
determinant of the information matrix for invariant designs with
$\bar{w}_{U}=1-\bar{w}_{L}$.
\begin{theorem}
  \label{thm:Dopt2Orbit}
  Let $\bar{w}_{L}^*= 1/2$ if $L+U=K$ and 
  \begin{multline}
    \label{eq:weightDopt2Orbit}
    \bar{w}_{L}^*=\frac{\left(U-L\right)\left(L+U-K\right)K -2U(K-U)}%
    {2\left(U-L\right)\left(L+U-K\right)(K+1)}\\
    \qquad+\frac{\sqrt{\left(U-L\right)^2\left(L+U-K\right)^2K^2
        +4L(K-L)U(K-U)}}%
    {2\left(U-L\right)\left(L+U-K\right) (K+1)}
  \end{multline}
  otherwise.
  
  In the case
  \begin{equation}
    \label{eq:cond2Orbit}
    (K-2L)(2U-K) < K
  \end{equation}
  the invariant design on the orbits $\mathcal{O}_{L}$ and
  $\mathcal{O}_{U}$ with weights $\bar{w}_{L}^*$ and
  $\bar{w}_{U}^* = 1-\bar{w}_{L}^*$ is $D$-optimal.
\end{theorem}
The weight $\bar{w}_{L}^*$ simplifies considerably for $L=0$ and does
not depend on $U$. In this case $\bar{w}_{L}^* = 1/(K + 1)$. This is
exemplified in Table~\ref{tab:opt2orbit}. Because of symmetry follows
$\bar{w}_{L}^* = K / (K + 1)$, if $U = K$.

Another property of the weight $\bar{w}_{L}^*$ as a function of $U$,
which is visible in the table, is the symmetry around $K/2$. For fixed
$L$, $K$ and some constant $c>0$ the weight for $U = K/2 + c$ is the
same as for $U = K/2 - c$. Taking into account that for
optimality~\eqref{eq:cond2Orbit} and $L<U$ must be satisfied, this is
especially relevant for $U$ close to $K/2$, e.g.\ for $K$ odd and
$U=(K \pm 1) /2$.

For symmetric constraints the following result is immediate.
\begin{corollary}
  \label{cor:thm:Dopt2OrbitSym}
  Let $L + U = K$.
  If
  \begin{equation*}
    \frac{K-\sqrt{K}}{2}<L
  \end{equation*}
  then the invariant design with $\bar{w}_{L}=\bar{w}_{U}=1/2$ is
  $D$-optimal.
\end{corollary}

The next result is concerned with designs on a sufficiently wide range
of orbits.

\begin{theorem}
  \label{thm:DoptEffAsFactorial}
  In the case
  \begin{equation}
    \label{eq:condEffAsFact}
    (K-2L)(2U-K) \geq K
    \,.
  \end{equation}
  an invariant design is $D$-optimal if and only if $m_{1}=m_{2}=0$.
\end{theorem}

For symmetric regions the result again simplifies:

\begin{corollary}
  \label{cor:thm:DoptEffAsFactorialSym}
  Let $L + U = K$.
  If
  \begin{equation}
    \label{eq:conditionCor1}
    L \leq \frac{K-\sqrt{K}}{2}
  \end{equation}
  then an invariant design is $D$-optimal if and only if $m_{2}=0$.
\end{corollary}

In the situation of Theorem \ref{thm:DoptEffAsFactorial} and Corollary
\ref{cor:thm:DoptEffAsFactorialSym} the information matrix of the
optimal design is the $p\times p$ identity matrix and coincides with
the information matrix of the $2^K$ factorial on the unrestricted
design region.

In the proof of Theorem~\ref{thm:DoptEffAsFactorial} given in the
appendix we show that exemplary designs $\bar{\xi}^*$ of the form
\begin{equation}
  \label{eq:optDesign3Orbits}
  \bar{\xi}^*=
  \left\{\begin{matrix}
      L&\ell&U\\
      \bar{w}_{L}^*&1-\bar{w}_{L}^*-\bar{w}_{U}^*&\bar{w}_{U}^*
    \end{matrix}\right\}
\end{equation}
fulfil the conditions of the theorem.  If $K$ is even $\ell= K/2$ may
be chosen for the interior orbit.  If $K$ is odd the choice depends on
$L$ and $U$, too.  If $L<(K-\sqrt{K})/2$ choose $\ell = (K-1)/2$.  If
$U>(K+\sqrt{K})/2$ choose $\ell = (K+1)/2$.  If both conditions are
met, we can choose any of the two given values.  These choices ensure,
that $L<\ell<U$.

On the boundary orbits the weights are
\begin{equation}
  \label{eq:optWeights3Orbit}
  \bar{w}_{L}^*=\frac{K+(2\ell-K)(2U-K)}{4(\ell-L)(U-L)}
  \quad\text{and}\quad
  \bar{w}_{U}^*=\frac{K+(2L-K)(2\ell-K)}{4(U-\ell)(U-L)}\,.
\end{equation}
Condition~\eqref{eq:condEffAsFact} guarantees, that the weight of the
interior orbit is non-negative. If equality holds
in~\eqref{eq:condEffAsFact}, then the middle weight is $0$, and a two
orbit design on $\mathcal{O}_{L}$ and $\mathcal{O}_{U}$ with weights
\begin{equation*}
  \bar{w}_{L}^*=  \frac{2U-K}{2(U-L)}
  \quad\text{and}\quad
  \bar{w}_{U}^*=  \frac{K-2L}{2(U-L)}
\end{equation*}
is optimal. Examples for these designs are given in
Table~\ref{tab:opt3orbit}.

Note that under the conditions of Theorem~\ref{thm:DoptEffAsFactorial}
the optimal design is not necessarily unique.  The weights for a
general optimal three orbit design with orbits
$\mathcal{O}_{\tilde{L}}$, $\mathcal{O}_{\tilde{\ell}}$ and
$\mathcal{O}_{\tilde{U}}$,
$L\leq \tilde{L} <\tilde{\ell}<\tilde{U}\leq U$ can be calculated by
substituting $L$ and $U$ with $\tilde{L}$ and $\tilde{U}$,
respectively, in~\eqref{eq:optWeights3Orbit}.
Condition~\eqref{eq:condEffAsFact} of
Theorem~\ref{thm:DoptEffAsFactorial} is replaced by
\begin{align*}
  (K-2\tilde{L})(2\tilde{U}-K) &\geq K\\
  (2\tilde{\ell}-K)(2\tilde{U}-K) &\geq -K\\
  (2\tilde{L}-K)(2\tilde{\ell}-K) &\geq - K\,.
\end{align*}
Again these conditions ensure, that the weights are non-negative.

An example for a symmetric optimal design, still under the conditions
of Theorem~\ref{thm:DoptEffAsFactorial}, is given by
\begin{equation*}
  \left\{
    \begin{matrix}
      k_{1}^*&k_{2}^*&K-k_{2}^*&K-k_{1}^*\\
      \bar{w}_{1}^*&\bar{w}_{2}^*&\bar{w}_{2}^*&\bar{w}_{1}^*
    \end{matrix}
  \right\}
\end{equation*}
with weights
\begin{equation*}
  \bar{w}_{1}^* = \frac{K-(K-2k_{2}^*)^2}{8(k_{2}^*-k_{1}^*)(K-k_{1}^*-k_{2}^*)}
  \quad\text{and}\quad
  \bar{w}_{2}^*=\frac{1-2\bar{w}_{1}^*}{2}
\end{equation*}
for any  $k_{1}^*, k_{2}^*$ satisfying
\begin{equation*}
  L\leq k_{1}^* < \frac{K-\sqrt{K}}{2} \leq k_{2}^* \leq \frac{K}{2}
  \quad\text{and}\quad
  K-k_{1}^*\leq U\,. 
\end{equation*}
If $k_{2}^* = (K-\sqrt{K})/2$ the weight $\bar{w}_{1}^*= 0$ and the
design reduces to a symmetric design on two orbits
$\mathcal{O}_{k_2^*}$ and $\mathcal{O}_{K-k_2^*}$ with
$\bar{w}_{2}^* = 1/2$. For $k_{2}^* = K/2$ it becomes a symmetric
three-orbit design.

Note especially, that designs with symmetric orbits can be optimal in
the case of asymmetric restrictions on the design region and vice
versa.  See for example Table~\ref{tab:opt4orbit}.

\section{Concluding Remarks}
It is noteworthy, that the above mentioned three and four orbit
designs have rational weights and hence can be implemented in practice
quite easily. For the two orbit designs from
Theorem~\ref{thm:Dopt2Orbit} this is not always the case.  An
affirmative example for $K=6$ rules and bounds $L=2$, $U=4$ with $30$
observations is given in Table~\ref{tab:desmatopt}.

The optimal designs in Theorem~\ref{thm:DoptEffAsFactorial} result in
an information matrix equal to the identity and are, hence, as
efficient as the full factorial design. Therefore they are also
optimal for other optimality criteria like the $A$-criterion for
minimising the average variance of the parameter estimates, the
$E$-criterion of maximising the smallest eigenvalue of the information
matrix, or the general class of Kiefer's $\Phi_q$-criteria based on
the eigenvalues of the information matrix \citep[see
e.g.][]{Pukelsheim:1993}.

In some particular cases of wide margins the optimal designs turn out
to be regular fractional factorial designs.  Consider for example
$K=4$, $L=1$ and $U=3$. An optimal design is given by the orbits
$\mathcal{O}_{1}$ and $\mathcal{O}_{3}$ with $\bar{w}_{1} = 1/2$ which
form an $2^{4-1}$ fractional factorial design.  In general orthogonal
arrays can occur. The symmetric four orbit design in the case $K=5$
with $k_{1}^*=1$ and $k_{2}^*=2$ is given by the first columns of an
$OA(40, 2^{20})$ \citep{SloanWeb}. While orthogonal arrays appear
naturally in these cases, further studies are necessary to explore the
specific relationship.  Further work has also to be done to generalise
the present results to models incorporating interactions between the
rules.

\section*{Acknowledgement}
This work was partly supported under DFG grant HO 1286/6-4 and SCHW
531/15-4, while the first author was affiliated with the University of
Magdeburg.

\section*{Appendix A: Proofs}
\begin{proof}[Proof of Lemma~\ref{lem:regularInfo}]
  The information matrix is regular if and only if its determinant is
  positive. Since the factors of the determinant in~\eqref{eq:DetOfM}
  are non-negative, this is equivalent to both of the factors being
  positive.  Consider
  \begin{equation*}
    1 + (K-1)m_{2} - Km_{1}^2
    = \frac{4}{K}\sum_{k=L}^{U}\bar{w}_{k}\left(k
      - \left(\sum_{\ell=L}^{U}\bar{w}_{\ell}\ell\right)\right)^2\geq 0\,,
  \end{equation*}
  which is $0$ if and only if $\bar{w}_{k}=1$ for some
  $k\in\{L,\ldots,U\}$.  Hence there have to be at least two orbits
  with positive weight.
  
  For $K\geq 2$ note that 
  \begin{equation*}
    1-m_{2}
    = \frac{4}{K(K-1)}\sum_{k=L}^{U}\bar{w}_{k}k(K-k) \geq 0\,.
  \end{equation*}
  This expression is equal to $0$ if and only if $\bar{w}_{k}=0$ for
  all $k\in\{1,\ldots,K-1\}\cap\{L\ldots,U\}$.  Hence the Lemma
  follows.
\end{proof}
Before we give the proof for Theorem~\ref{thm:Dopt2Orbit} we introduce
the following auxiliary result:
\begin{lemma}
  \label{lem:sensitivitypoly}
  Let $\bar{\xi}$ be an invariant design.  Then the sensitivity
  function
  $\psi(\bm{x})=\bm{f}(\bm{x})^{\top}\mathbf{M}(\bar{\xi})^{-1}\bm{f}(\bm{x})$
  is constant on the orbits, $\psi(\bm{x})=\tilde{\psi}(k)$ for
  $\bm{x}\in\mathcal{O}_k$, say, and the function $\tilde{\psi}$ is a
  polynomial in $k$ of degree at most $2$.
\end{lemma}
\begin{proof}[Proof of Lemma~\ref{lem:sensitivitypoly}]
  In our case the sensitivity $\psi$ is given by
  \begin{equation}
    \label{eq:sensitivity}
    \psi(\bm{x})
    = a_{0} +a_{1} \bm{1}_{K}^\top\bm{x}+a_{2} (\bm{1}_{K}^\top\bm{x})^{2}
  \end{equation}
  with the following coefficients:
  \begin{align*}
    a_{0}&=\frac{1+(K-1)m_{2}}{1+(K-1)m_{2} - Km_{1}^2}
           +\frac{K}{1-m_{2}}\,,
    &a_{1}&=- \frac{2m_{1}}{1+(K-1)m_{2} - Km_{1}^2}\,,\\
    a_{2}&=\frac{m_{1}^2-m_{2}}{(1-m_{2})(1+(K-1)m_{2} - Km_{1}^2)}\,.
  \end{align*}
  Since $\bm{1}_{K}^\top\bm{x}=2k-K$ for $\bm{x}\in \mathcal{O}_{k}$,
  the sensitivity function $\psi$ is constant on the orbits and
  $\tilde{\psi}(k)= a_0+a_1(2k-K)+a_2(2k-K)^2$, for
  $\bm{x}\in\mathcal{O}_{k}$, is a polynomial of degree at most~$2$.
\end{proof}
\begin{lemma}
  \label{lem:m1m2not0}
  Let $(K-2L)(2U-K) < K$, then $m_{1} \neq 0$ or $m_{2} \neq 0$ for
  every invariant design $\bar{\xi}$.
\end{lemma}
\begin{proof}[Proof of Lemma~\ref{lem:m1m2not0}]
  We prove the lemma by contradiction.
 
  Let $m_{1}=m_{2}=0$, then $\sum_{k=L}^{U}\bar{w}_{k}(2k-K) =0$
  and
  \begin{equation}
    \label{eq:2}
    \sum_{k=L}^{U}\bar{w}_{k}(2k-K)^2 = K\,.
  \end{equation}
  These sums can be seen as the mean and variance of a
  discrete zero-mean random variable taking values in
  $\{2L-K,\ldots,2U-K\}$.  Using the inequality in
  \citet{BhatiaDavis:2000} the variance is bounded above:
  \begin{equation*}
    \sum_{k=L}^{U}\bar{w}_{k}(2k-K)^2
    \leq (K-2L)(2U-K)\,.
  \end{equation*}
  Since $(K-2L)(2U-K) < K$ this is a contradiction to
  equation~\eqref{eq:2}.  
\end{proof}
\begin{proof}[Proof of Theorem~\ref{thm:Dopt2Orbit}]
  We will show first, that the optimal design has to be concentrated
  on the two boundary orbits.
  
  For an optimal design $\xi^*$ the equivalence
  theorem~\citep{KieferWolfowitz:1960} yields
  \begin{equation}
    \label{eq:DEquivalence}
    \psi(\bm{x}) =
    \bm{f}(\bm{x})^\top\mathbf{M}(\bar{\xi}^*)^{-1}\bm{f}(\bm{x})  \leq p
  \end{equation}
  for all $\bm{x}\in\mathcal{X}$. For an optimal invariant design this
  can be written equivalently as $\tilde{\psi}(k)\leq p$ for all
  $k\in\{L,\ldots,U\}$.  Consider the leading coefficient $a_{2}$ of
  $\tilde{\psi}$.

  Let $a_{2}\leq 0$, then, following from the equivalence theorem,
  either $\tilde{\psi}(k)=p$ for exactly one $k\in\{L,\ldots,U\}$ or
  $\tilde{\psi}(k)=\tilde{\psi}(k+1)=p$ for some
  $k\in\{L,\ldots,U-1\}$.

  In the latter case $\tilde{\psi}(k)\leq p$ for all $k\in\{0,\ldots,K\}$ and
  thus the design would be optimal not only on the design region
  $\mathcal{X}$ but on the whole set $\{-1,+1\}^{K}$. It follows, that
  the information matrix is the identity matrix.  But this
  contradicts, that by Lemma~\ref{lem:m1m2not0} either $m_{1}\neq 0$
  or $m_{2}\neq 0$.
  
  In the first case the optimal design would be concentrated on the
  orbit $\mathcal{O}_{k}$, which leads to a singular information
  matrix and consequently to a contradiction.
  
  Hence, $a_{2}>0$ and $\tilde{\psi}$ attains its maximum (equal to
  $p$) on the boundary, i.e.\ the optimal design is concentrated on
  the orbits $\mathcal{O}_{L}$ and $\mathcal{O}_{U}$.

  In order to obtain the optimal design, it remains to find the optimal
  weight $w_L^*$ on $\mathcal{O}_{L}$ (and consequently
  $w_U^*=1-w_L^*$ on $\mathcal{O}_U$). Optimizing the determinant then
  yields the optimal weight $w_L^*$ specified in the theorem.

\end{proof}

\begin{proof}[Proof of Theorem~\ref{thm:DoptEffAsFactorial}]
  Since any invariant design $\bar{\xi}$ on $\mathcal{X}$ with
  $m_{1}=m_{2}=0$ is optimal on the unrestricted design region
  $\{-1,+1\}^K$, by majorization, these designs are also optimal on
  the design region $\mathcal{X}$.
  
  For $K=1$ there is nothing to show.  For $K\geq 2$ we will show that
  the design in~\eqref{eq:optDesign3Orbits} yields $m_{1}=m_{2}=0$
  and hence is optimal.

  It follows from condition~\eqref{eq:condEffAsFact}
  of the theorem, that
  \begin{equation*}
    L< \frac{K-\sqrt{K}}{2} \quad\text{or}\quad U>\frac{K+\sqrt{K}}{2}
  \end{equation*}
  and, that $L < K/2 < U$. Hence we can choose the interior orbit
  $\mathcal{O}_{\ell}$ as described after~\eqref{eq:optDesign3Orbits}.
  The weights are non-negative by the choice of $\ell$, since
  $L<\ell<U$. For the first moment $m_{1}$ it follows that
  \begin{multline*}
    m_{1}
    = \frac{2(L-\ell)}{K}\bar{w}_{L}^*
      +\frac{2\ell-K}{K}
      +\frac{2(U-\ell)}{K}\bar{w}_{U}^*\\
    =\frac{(2\ell-K)(-2U+K+2U-2L+2L-K)}{2K(U-L)} = 0\,.
  \end{multline*}
  The second quantity $m_{2}$ can be written as
  \begin{equation*}
    m_{2}
    = \frac{4(\ell-L)(K-L-\ell)}{K(K-1)}\bar{w}_{L}^*
    +\frac{(2\ell-K)^2-K}{K(K-1)}
    +\frac{4(U-\ell)(U+\ell-K)}{K(K-1)}\bar{w}_{U}^*\,.
  \end{equation*}
  Substituting the weights yields
  \begin{multline*}
    m_{2}=\frac{(K-L-\ell)(K+(2\ell-K)(2U-K))
      +((2\ell-K)^2-K)(U-L)}{K(K-1)(U-L)}\\
    +\frac{(U+\ell-K)(K+(2\ell-K)(2L-K))}{K(K-1)(U-L)}=0\,.
  \end{multline*}
  Hence $m_{1}=m_{2}=0$ and the given design is $D$-optimal, which
  concludes the proof.
\end{proof}

\newpage
\section*{Appendix B: Tables}
\begin{table}[H]
\centering
\caption{Examples for optimal invariant two orbit designs from
  Theorem~\ref{thm:Dopt2Orbit} and their efficiency with respect to
  the full factorial design}
\label{tab:opt2orbit}
\fbox{%
\begin{tabular}{rrrrrrrr}
  $K$ & $\frac{K-\sqrt{K}}{2}$ & $\frac{K+\sqrt{K}}{2}$ & $L$ & $U$
  & $\bar{w}_{L}^*$ & $\bar{w}_{U}^*$ & Efficiency \\ 
  \hline
  $2$ & $0.29$ & $1.71$ & $0$ & $1$ & $0.3333$ & $0.6667$ & $0.8399$ \\ 
  \hline
  $3$ & $0.63$ & $2.37$ & $0$ & $1$ & $0.2500$ & $0.7500$ & $0.7071$ \\ 
  \cline{4-8}
      & & & $1$ & $2$ & $0.5000$ & $0.5000$ & $0.8774$ \\ 
  \hline
  $4$ & $1.00$ & $3.00$ & $0$ & $1$ & $0.2000$ & $0.8000$ & $0.6063$ \\ 
      & & & $0$ & $2$ & $0.2000$ & $0.8000$ & $0.9507$ \\ 
  \cline{4-8}
      & & & $1$ & $2$ & $0.4000$ & $0.6000$ & $0.8386$ \\ 
  \hline
  $5$ & $1.38$ & $3.62$ & $0$ & $1$ & $0.1667$ & $0.8333$ & $0.5291$ \\ 
      & & & $0$ & $2$ & $0.1667$ & $0.8333$ & $0.8736$ \\ 
  \cline{4-8}
      & & & $1$ & $2$ & $0.3333$ & $0.6667$ & $0.7828$ \\ 
      & & & $1$ & $3$ & $0.3333$ & $0.6667$ & $0.9863$ \\ 
  \cline{4-8}
      & & & $2$ & $3$ & $0.5000$ & $0.5000$ & $0.8636$ \\ 
  \hline
  $6$ & $1.78$ & $4.22$ & $0$ & $1$ & $0.1429$ & $0.8571$ & $0.4688$ \\ 
      & & & $0$ & $2$ & $0.1429$ & $0.8571$ & $0.7994$ \\ 
      & & & $0$ & $3$ & $0.1429$ & $0.8571$ & $0.9764$ \\ 
  \cline{4-8}
      & & & $1$ & $2$ & $0.2857$ & $0.7143$ & $0.7263$ \\ 
      & & & $1$ & $3$ & $0.2590$ & $0.7410$ & $0.9486$ \\ 
  \cline{4-8}
      & & & $2$ & $3$ & $0.4286$ & $0.5714$ & $0.8491$ \\ 
      & & & $2$ & $4$ & $0.5000$ & $0.5000$ & $0.9882$ \\ 
  \hline
  $9$ & $3.00$ & $6.00$ & $0$ & $1$ & $0.1000$ & $0.9000$ & $0.3482$ \\ 
      & & & $0$ & $2$ & $0.1000$ & $0.9000$ & $0.6259$ \\ 
      & & & $0$ & $3$ & $0.1000$ & $0.9000$ & $0.8299$ \\ 
      & & & $0$ & $4$ & $0.1000$ & $0.9000$ & $0.9564$ \\ 
  \cline{4-8}
      & & & $1$ & $2$ & $0.2000$ & $0.8000$ & $0.5844$ \\ 
      & & & $1$ & $3$ & $0.1643$ & $0.8357$ & $0.8045$ \\ 
      & & & $1$ & $4$ & $0.1545$ & $0.8455$ & $0.9432$ \\ 
      & & & $1$ & $5$ & $0.1545$ & $0.8455$ & $0.9991$ \\ 
  \cline{4-8}
      & & & $2$ & $3$ & $0.3000$ & $0.7000$ & $0.7465$ \\ 
      & & & $2$ & $4$ & $0.2539$ & $0.7461$ & $0.9158$ \\ 
      & & & $2$ & $5$ & $0.2539$ & $0.7461$ & $0.9932$ \\ 
  \cline{4-8}
      & & & $3$ & $4$ & $0.4000$ & $0.6000$ & $0.8418$ \\ 
      & & & $3$ & $5$ & $0.4000$ & $0.6000$ & $0.9670$ \\ 
  \cline{4-8}
      & & & $4$ & $5$ & $0.5000$ & $0.5000$ & $0.8733$ \\ 
\end{tabular}}
\end{table}

\begin{table}[ht]
\centering
\caption{Optimal invariant three orbit designs from
  Theorem~\ref{thm:DoptEffAsFactorial}}
\label{tab:opt3orbit}
\fbox{%
\begin{tabular}{rrrrrrr}
  $K$ & $L$ & $U$ & $\ell$ & $\bar{w}_{L}^*$ & $\bar{w}_{U}^*$ & $\bar{w}_{\ell}^*$ \\ 
  \hline
  2 & 0 & 2 & 1 & 0.2500 & 0.2500 & 0.5000 \\ 
  \hline
  3 & 0 & 2 & 1 & 0.2500 & 0.7500 &   $-$ \\ 
      & 0 & 3 & 1 &   $-$ & 0.2500 & 0.7500 \\ 
  \hline
  4 & 0 & 3 & 2 & 0.1667 & 0.3333 & 0.5000 \\ 
      & 0 & 4 & 2 & 0.1250 & 0.1250 & 0.7500 \\ 
  \cline{2-7} & 1 & 3 & 2 & 0.5000 & 0.5000 &   $-$ \\ 
  \hline
  5 & 0 & 3 & 2 & 0.1667 & 0.8333 &   $-$ \\ 
      & 0 & 4 & 2 & 0.0625 & 0.3125 & 0.6250 \\ 
      & 0 & 5 & 2 &   $-$ & 0.1667 & 0.8333 \\ 
  \cline{2-7} & 1 & 4 & 2 & 0.1667 & 0.3333 & 0.5000 \\ 
  \hline
  6 & 0 & 4 & 3 & 0.1250 & 0.3750 & 0.5000 \\ 
      & 0 & 5 & 3 & 0.1000 & 0.1500 & 0.7500 \\ 
      & 0 & 6 & 3 & 0.0833 & 0.0833 & 0.8333 \\ 
  \cline{2-7} & 1 & 4 & 3 & 0.2500 & 0.5000 & 0.2500 \\ 
      & 1 & 5 & 3 & 0.1875 & 0.1875 & 0.6250 \\ 
  \hline
  9 & 0 & 5 & 4 & 0.1000 & 0.9000 &   $-$ \\ 
      & 0 & 6 & 4 & 0.0625 & 0.3750 & 0.5625 \\ 
      & 0 & 7 & 4 & 0.0357 & 0.2143 & 0.7500 \\ 
      & 0 & 8 & 4 & 0.0156 & 0.1406 & 0.8438 \\ 
      & 0 & 9 & 4 &   $-$ & 0.1000 & 0.9000 \\ 
  \cline{2-7} & 1 & 6 & 4 & 0.1000 & 0.4000 & 0.5000 \\ 
      & 1 & 7 & 4 & 0.0556 & 0.2222 & 0.7222 \\ 
      & 1 & 8 & 4 & 0.0238 & 0.1429 & 0.8333 \\ 
  \cline{2-7} & 2 & 6 & 4 & 0.1875 & 0.4375 & 0.3750 \\ 
      & 2 & 7 & 4 & 0.1000 & 0.2333 & 0.6667 \\ 
  \cline{2-7} & 3 & 6 & 4 & 0.5000 & 0.5000 &   $-$ \\ 
\end{tabular}}
\end{table}
\begin{table}[ht]
\centering
\caption{Optimal symmetric invariant four orbit designs from
  Theorem~\ref{thm:DoptEffAsFactorial}}
\label{tab:opt4orbit}
\fbox{%
\begin{tabular}{rrrrrrrrrrr}
  $K$ & $L$ & $U$ & $k_{1}^*$ & $k_{2}^*$ & $k_{3}^*$ & $k_{4}^*$
  & $\bar{w}_{1}^*$ & $\bar{w}_{2}^*$ & $\bar{w}_{3}^*$ & $\bar{w}_{4}^*$ \\ 
  \hline
  3 & 0 & 3 & 0 & 1 & 2 & 3& 0.1250 & 0.3750  & 0.3750 & 0.1250\\ 
  \hline
  4 & 0 & 4 & 0 & 1 & 3 & 4 & $-$ & 0.5000 & 0.5000 & $-$ \\ 
  \hline
  5 & 0 & 4 & 1 & 2 & 3 & 4 & 0.2500 & 0.2500 & 0.2500 & 0.2500 \\ 
    & 0 & 5 & 0 & 2 & 3 & 5 & 0.0833 & 0.4167 & 0.4167 & 0.0833 \\ 
    & 0 & 5 & 1 & 2 & 3 & 4 & 0.2500 & 0.2500 & 0.2500 & 0.2500 \\ 
  \cline{2-11}                                           
    & 1 & 4 & 1 & 2 & 3 & 4 & 0.2500 & 0.2500 & 0.2500 & 0.2500 \\ 
  \hline                                                
  6 & 0 & 5 & 1 & 2 & 4 & 5 & 0.0833 & 0.4167 & 0.4167 & 0.0833 \\ 
    & 0 & 6 & 0 & 2 & 4 & 6 & 0.0312 & 0.4688 & 0.4688 & 0.0312 \\ 
    & 0 & 6 & 1 & 2 & 4 & 5 & 0.0833 & 0.4167 & 0.4167 & 0.0833 \\ 
  \cline{2-11}                                           
    & 1 & 5 & 1 & 2 &  4 & 5 & 0.0833 & 0.4167 & 0.4167 & 0.0833 \\ 
  \hline                                                
  9 & 0 & 7 & 2 & 4 & 5 & 7 & 0.1667 & 0.3333 & 0.3333 & 0.1667 \\ 
    & 0 & 8 & 1 & 4 & 5 & 8 & 0.0833 & 0.4167 & 0.4167 & 0.0833 \\ 
    & 0 & 8 & 2 & 4 & 5 & 7 & 0.1667 & 0.3333 & 0.3333 & 0.1667 \\ 
    & 0 & 9 & 0 & 4 & 5 & 9 & 0.0500 & 0.4500 & 0.4500 & 0.0500 \\ 
    & 0 & 9 & 1 & 4 & 5 & 8 & 0.0833 & 0.4167 & 0.4167 & 0.0833 \\ 
    & 0 & 9 & 2 & 4 & 5 & 7 & 0.1667 & 0.3333 & 0.3333 & 0.1667 \\ 
  \cline{2-11}                                           
    & 1 & 7 & 2 & 4 & 5 & 7 & 0.1667 & 0.3333 & 0.3333 & 0.1667 \\ 
    & 1 & 8 & 1 & 4 & 5 & 8 & 0.0833 & 0.4167 & 0.4167 & 0.0833 \\ 
    & 1 & 8 & 2 & 4 & 5 & 7 & 0.1667 & 0.3333 & 0.3333 & 0.1667 \\ 
  \cline{2-11}                                           
    & 2 & 7 & 2 & 4 & 5 & 7 & 0.1667 & 0.3333 & 0.3333 & 0.1667 \\ 
\end{tabular}}
\end{table}

\begin{table}[ht]
\centering
\caption{Design for the invariant optimal two orbit design for $K=6$
  with $L=2$ and $U=4$, $+$ an $-$ denote $+1$ and $-1$, respectively}
\label{tab:desmatopt}
\fbox{%
\begin{tabular}{lllllll}
   $+$ & $+$ & $-$ & $-$ & $-$ & $-$ \\ 
   $+$ & $-$ & $+$ & $-$ & $-$ & $-$ \\ 
   $-$ & $+$ & $+$ & $-$ & $-$ & $-$ \\ 
   $+$ & $-$ & $-$ & $+$ & $-$ & $-$ \\ 
   $-$ & $+$ & $-$ & $+$ & $-$ & $-$ \\ 
   $-$ & $-$ & $+$ & $+$ & $-$ & $-$ \\ 
   $+$ & $-$ & $-$ & $-$ & $+$ & $-$ \\ 
   $-$ & $+$ & $-$ & $-$ & $+$ & $-$ \\ 
   $-$ & $-$ & $+$ & $-$ & $+$ & $-$ \\ 
   $-$ & $-$ & $-$ & $+$ & $+$ & $-$ \\ 
   $+$ & $-$ & $-$ & $-$ & $-$ & $+$ \\ 
   $-$ & $+$ & $-$ & $-$ & $-$ & $+$ \\ 
   $-$ & $-$ & $+$ & $-$ & $-$ & $+$ \\ 
   $-$ & $-$ & $-$ & $+$ & $-$ & $+$ \\ 
   $-$ & $-$ & $-$ & $-$ & $+$ & $+$ \\ 
   $+$ & $+$ & $+$ & $+$ & $-$ & $-$ \\ 
   $+$ & $+$ & $+$ & $-$ & $+$ & $-$ \\ 
   $+$ & $+$ & $-$ & $+$ & $+$ & $-$ \\ 
   $+$ & $-$ & $+$ & $+$ & $+$ & $-$ \\ 
   $-$ & $+$ & $+$ & $+$ & $+$ & $-$ \\ 
   $+$ & $+$ & $+$ & $-$ & $-$ & $+$ \\ 
   $+$ & $+$ & $-$ & $+$ & $-$ & $+$ \\ 
   $+$ & $-$ & $+$ & $+$ & $-$ & $+$ \\ 
   $-$ & $+$ & $+$ & $+$ & $-$ & $+$ \\ 
   $+$ & $+$ & $-$ & $-$ & $+$ & $+$ \\ 
   $+$ & $-$ & $+$ & $-$ & $+$ & $+$ \\ 
   $-$ & $+$ & $+$ & $-$ & $+$ & $+$ \\ 
   $+$ & $-$ & $-$ & $+$ & $+$ & $+$ \\ 
   $-$ & $+$ & $-$ & $+$ & $+$ & $+$ \\ 
   $-$ & $-$ & $+$ & $+$ & $+$ & $+$ \\ 
\end{tabular}}
\end{table}
\end{document}